\documentclass[10pt,reqno]{amsart}
\usepackage[utf8]{inputenc}
\usepackage[english]{babel}
\usepackage{amsmath, amssymb, amsthm}
\usepackage{mathrsfs}
\usepackage{enumitem}
\usepackage{graphicx}
\usepackage{tikz-cd}
\usepackage{xcolor}
\numberwithin{equation}{section} 

\usepackage{diagbox} 
\usepackage{tensor}  

\usepackage{hyperref} 

\usepackage[alphabetic]{amsrefs} 

\theoremstyle{plain}
\newtheorem{thm}{Theorem}[section]

\newtheorem{lemma}[thm]{Lemma}

\theoremstyle{definition}
\newtheorem{defin}[thm]{Definition}

\newtheorem{rmk}[thm]{Remark}

\allowdisplaybreaks

\newcommand{\N}{\mathbb{N}}

\newcommand{\R}{\mathbb{R}}

\newcommand{\abs}[1]{\left|#1 \right|}
\newcommand{\norm}[1]{\left\lVert#1\right\rVert}
\DeclareMathOperator{\supp}{supp}

\newcommand{\eps}{\varepsilon}

\newcommand{\MP}{\mathcal{MP}}

\def\XXint#1#2#3{{\setbox0=\hbox{$#1{#2#3}{\int}$ }
\vcenter{\hbox{$#2#3$ }}\kern-.6\wd0}}

\usepackage{scalerel,stackengine}
\stackMath
\newcommand\hhat[1]{%
\savestack{\tmpbox}{\stretchto{%
  \scaleto{%
    \scalerel*[\widthof{\ensuremath{#1}}]{\kern.1pt\mathchar"0362\kern.1pt}%
    {\rule{0ex}{\textheight}}
  }{\textheight}%
}{2.4ex}}%
\stackon[-6.9pt]{#1}{\tmpbox}%
}

\DeclareMathOperator{\Op}{Op}
\DeclareMathOperator{\wf}{WF}

\title[(In)stability results for the Light Ray Transform]{(In)stability results for the Light Ray Transform}

\author{Leonard Busch}
\address{The School of Mathematics \& Statistics, The University of Melbourne, Melbourne, Australia}
\email{lbusch@unimelb.edu.au}

\author{Sebastián Muñoz-Thon}
\address{Universit\'e Paris-Saclay, Laboratoire de math\'ematiques d’Orsay, 91405, Orsay, France.}
\email{sebastian.munoz-thon@universite-paris-saclay.fr}

\author{Lauri Oksanen}
\address{Department of Mathematics and Statistics, University of Helsinki, PO BOX 68, 00014, HELSINKI, Finland.}
\email{lauri.oksanen@helsinki.fi}

\begin{document}

\begin{abstract}
We prove quantitative instability results for the light ray transform in the Minkowski setting, and locally in the smooth and Gevrey categories. In the smooth Sobolev setting, we show that no modulus of continuity can be better than $t^\alpha$ for any $\alpha \in(0,1)$, while in the Gevrey setting we obtain an explicit logarithmic lower bound. On the other hand, using the relation between stationary geometries and magnetic(-potential) systems, we obtain a stability result for ``time moments'' of the light ray transform.
\end{abstract}

\maketitle

\section{Introduction}

Besides its medical applications, the geodesic X-ray transform appears in geometric inverse problems as the linearization of the boundary rigidity problem \cite{PSU23}. Its stability and instability properties are well understood, see for instance \cites{SU04, PSU15, AS20, PS21, IP22, PSU23}. One of the analogs in the Lorentzian setting is the light ray transform, which besides appearing after linearizing the corresponding geometric inverse problem, also appears naturally in the context of the recovery of coefficients in wave equations. This operator integrates functions (or tensors) over light ray geodesics. It has been recently studied in \cites{Stefanov17, FIO21, Stefanov 24, OPS25, Wang21}. Some stability results appear in \cites{VW21, Wang 22}, where the authors obtain such results for solutions of wave equations. There is also a (log-)instability result implicitly in \cite{Aicha}*{Section 3.2}, as a key step in the proof of instability for the recovery of time dependent coefficients from the Dirichlet-to-Neumann map for wave equations in the flat case. However, the Lorentzian problem is far from being completely understood. 

In this work, we study the behavior of the light ray transform. Let us begin by stating one such result. We obtain quantitative instability for this operator in the smooth and the Gevrey category. For simplicity, in the introduction we will only deal with the flat case, i.e., the Minkowski spacetime. For the results for general smooth and Gevrey spacetimes, see \S\ref{section:smooth},\ref{section:gevrey}. Hence, let us consider for now the Minkowski space $(M,g)$, where $M=\R^{1+n}$, and $g=-dt^{2}+(dx^{1})^{2}+\cdots+(dx^{n})^{2}$. Consider $\kappa \in C^{\infty}(\R^{1+n} \times S^{n-1})$ non-vanishing. A pair $(z,\theta) \in \mathcal{M}=\R_{z}^{n} \times S_{\theta}^{n-1}$ defines a lightlike geodesic $s \mapsto (s,z+s\theta)$ ($s \in \R$). For $f \in \mathcal{E}'(\R^{1+n})$, we define the (weighted) light ray transform
\begin{equation} \label{eq:lightray}
    L_{\kappa}f(z,\theta)=\int_{\R}\kappa(s,z+s\theta,\theta)f(s,z+s\theta)ds.
\end{equation}
We let $H_{S}^{s}(M)$ denote the Sobolev space of order $s$ with functions supported in $S \subset M$, and $G^{\sigma}(M)$ denotes the Gevrey space of order $\sigma$, we refer to \S\ref{section:prelim} and Definition \ref{defin:localized_spaces} for further details. In \eqref{eq:lightray}, $\kappa$ will be smooth or Gevrey depending on the context, and positively homogeneous of degree $1$ in the fiber variable. Fix $R>0$ and set
\begin{equation} \label{eq:compacts}
    S=\overline{B_{R}(0)}\subset \mathbb R^{1+n}, \qquad Q=\overline{B_{\sqrt{2}R}^{n}(0)}\times S^{n-1} \subset \mathbb R^n\times S^{n-1}.
\end{equation}
We obtain the following instability estimate:
\begin{thm} \label{thm:instability}
    Let $r,\delta>0$, take $S$ and $Q$ as in \eqref{eq:compacts}, and define $K_r = \{f\in H^\delta_S(\R^{1+n})\colon \norm{f}_{H^{\delta}} \leq r\}$.
    \begin{enumerate}
        \item If $\omega$ is a modulus of continuity such that  
        \[
        \|f_{1}-f_{2}\|_{L^2(\R^{1+n})} \leq \omega (\|L_{\kappa} f_{1}-L_{\kappa} f_{2}\|_{L^{2}(\R^{n} \times S^{n-1})} ), \quad f_{1},f_{2} \in K_{r},
        \]
        then for every $\alpha\in(0,1)$ one has $\omega(t)\gtrsim t^\alpha$ for $t$ small.
        \item For $\sigma>1$, suppose that $\kappa \in G^\sigma(\R^{1+n} \times S^{n-1})$ and that $\omega$ is a modulus of continuity so that
        \[
        \norm{f_1-f_2}_{L^2(\R^{1+n})} \leq \omega(\norm{L_\kappa f_1 - L_\kappa f_2}_{L^2(\R^n\times S^{n-1})})\,,\quad f_1,f_2 \in K_r\,.
        \]
        For $t$ small we have $\omega(t) \gtrsim|\log t|^{-\frac{\delta((2n-1)\sigma+1)}{n+1}}$.
    \end{enumerate}
\end{thm}

Our proof appeals to the machinery developed in \cite{KRS21} for general instability results for inverse problems. To apply such results, we have to study wave front sets and mapping properties of $L$. We recall that the light ray transform is unstable because it cannot recover timelike singularities. Indeed, in the Minkowski setting the normal operator associated with the light ray transform is elliptic in the spacelike cone, whereas the transform is microlocally smoothing on timelike covectors. Thus, although spacelike singularities can be recovered (microlocally under suitable assumptions), timelike ones are essentially invisible to the data, see for instance \cite{SUbook}.

Our result gives a quantitative measure of this instability. More precisely, we show that no modulus of continuity for the inverse problem can be better than the lower bounds stated in Theorem \ref{thm:instability}. In the smooth Sobolev setting this rules out any H\"older stability estimate, while in the Gevrey setting it yields a logarithmic lower bound. 

As mentioned above, we also prove analogous results on general smooth and Gevrey Lorentzian manifolds. The flat case is presented here only to highlight the main mechanism in a simple geometric setting. The manifold versions require localizing the argument in suitable coordinate charts and replacing the global construction by microlocal constructions adapted to the timelike cone. These results are stated in Theorems \ref{thm:manCinftyinsta} and \ref{thm:gevrey-insta}.

On the other hand, we obtain stability results for the moments of the light ray transform in stationary geometries. We now give a simplified version of the statement, the full one is given in Theorem \ref{thm:moment-stability}. We recall that a stationary spacetime is a manifold of the form $M=\R \times N$ (where $N$ is a compact Riemannian manifold with boundary), equipped with a Lorentzian metric which does not depend on time. Due to its symmetries, to a stationary manifold we can relate a magnetic system on the transversal manifold $N$, and lightlike geodesics are related to magnetic (or magnetic-potential) geodesics on $N$, i.e., curves satisfying a twisted (by a closed 2-form) version of the geodesic equation, see \cites{Stefanov24, MT24}. The proper definitions of magnetic objects are given in \S\ref{subsection:magnetic}. As for Riemannian manifolds, there is a notion of simplicity for magnetic systems, which implies that given any two boundary points there exists a unique magnetic geodesic joining them \cite{DPSU07}. Let us denote by $\partial_{+}SN$ the set of incoming directions.

\begin{thm} \label{thm:moment-stability-intro}
Let $(M,g)$ be a stationary Lorentzian manifold, and assume that the associated magnetic system is simple. Then, for every $k \geq 0$, there exist constants $A_{k}>0$ such that for any $f \in C_{0}^{\infty}(M)$, we have
\[
\left\| \int_{\R}t^{k}f(t,\bullet)dt \right\|_{L^{2}(N)} \leq A_{k} \sum_{r=0}^{k} \left\| \int_{\R} T^{r}L_{\gamma_{T,X,V}}f(T,\bullet)dT \right\|_{H^{1/2}(\partial_{+}SN)}.
\]
\end{thm}

Although the light ray transform does not stably recover $f$ itself, our result shows that its time moments of order $k$ are stably determined by the corresponding moments of order $0$ to $k$ of the light ray transform data.

\subsection*{Organization of the article} We begin Section \ref{section:prelim} recalling some notation and the result from \cite{KRS21} that we shall use here. In \S\ref{section:minkowski} we prove Theorem \ref{thm:instability}. We then obtain the analogous results in the local smooth and Gevrey category in Sections \ref{section:smooth} and \ref{section:gevrey}, respectively. In \S\ref{section:stability} we prove Theorem \ref{thm:moment-stability-intro}. Finally, we have two appendices: the first one \S\ref{appendix:gevrey} gives a result on Fourier integral operators in the Gevrey setting, while the second one \S\ref{appendix:magnetic} recalls concepts of magnetic(-potential) systems, and gives a stability result for the magnetic (and potential) X-ray transform.

\subsection*{Acknowledgments} S.~M.-T. was supported by the European Research Council (ERC) under the European Union’s Horizon 2020 research and innovation programme (Grant agreement no. 101162990 -- ADG). L.O. was supported by the European Research Council of the European Union, grant 101086697 (LoCal), and the Research Council of Finland, grants 347715 and 359182 (Flagship of Advanced Mathematics for Sensing Imaging and Modelling).

Views and opinions expressed are those of the authors only and do not
necessarily reflect those of the European Union or the other funding
organizations.

\section{Preliminaries} \label{section:prelim}

Let us begin with the definition of Sobolev spaces, we refer also to \cite{KRS21}. Consider $(M,g)$ a closed Riemannian manifold. For any $s \in \R$, the Bessel potential $J^{s}:=(1-\Delta_{g})^{s/2} \in \Psi^{s}(M)$ allows us to define the Sobolev norm by $\|u\|_{H^{s}(M)}=\|J^{s}u\|_{L^{2}(M)}$. Equivalently, let $(\varphi_{j})_{j=1}^{\infty}$ be an orthonormal basis of $L^{2}(M)$ consisting of eigenfunctions of $-\Delta_{g}$, with associated eigenvalues $(\lambda_{j})_{j=1}^{\infty}$. Using Weyl asymptotics $1+\lambda_{j} \sim j^{\frac{2}{n}}$ ($j \geq 1$), it can be shown that 
\[ \sum_{j=1}^{\infty}j^{\frac{2s}{n}}|(u,\varphi_{j})_{L^{2}(M)}|^{2}, \]
defines an equivalent norm on $H^{s}(M)$. 

Using the same sequence of eigenfunctions as before, for $1 \leq \sigma <\infty$ and $\rho>0$, define the norm 
\[
\|u\|_{A^{\sigma, \rho}(M)}=\left(\sum_{j=1}^{\infty} e^{2 \rho j^{\frac{1}{n \sigma}}}|(u, \varphi_j)_{L^{2}(M)}|^{2}\right)^{1/2}.
\]
Define the Hilbert space $A^{\sigma,\rho}(M)$ as the subspace of $L^{2}(M)$ of elements that have finite $\| \bullet \|_{A^{\sigma,\rho}(M)}$-norm. These spaces are significant due to their connection to spaces of Gevrey functions.

\begin{defin}
    For $n\in \mathbb{N}$ and $W\subset \R^n$ open and $\sigma\geq 1$, we define $G^\sigma(W)$ as the set of $f\in C^\infty(W;\mathbb{C})$ so that for every compact $K\subset W$ there is $C>0$ so that for all multi-indices $\alpha$,
    \[
        \sup_K \abs{\partial^\alpha f} \leq C^{1+\abs{\alpha}}\alpha!^\sigma\,.
    \]
    Furthermore, if $W' \subset \R^{n'}$ with $n'\in\mathbb{N}$ is open, then we write $f \in G^\sigma(W,W')$ if and only if each component of $f$ is in $G^\sigma(W)$. We call such $f$ Gevrey-$\sigma$ or $G^\sigma$ functions.
\end{defin}
When $\sigma=1$, the above reduces to the notion of analytic functions.

Following \cite{fd}*{Proposition~2.16}, we introduce
\begin{defin}
    Let $N$ be a smooth manifold and $\sigma\geq 1$. We call $N$ a Gevrey-$\sigma$ (or $G^\sigma$) manifold if and only if there is an atlas $\mathcal{A}$ of $N$ so that $\varphi \circ \psi^{-1}$ is a $G^\sigma$ function on its Euclidean domain for all $\varphi,\psi \in \mathcal{A}$. 

    For $N_1, N_2$ two $G^\sigma$ manifolds, we call a function $u \colon N_1 \to N_2$ Gevrey-$\sigma$ and write $u\in G^\sigma(N_1,N_2)$ if and only if $\psi \circ u \circ \varphi^{-1}$ is $G^\sigma$ on its Euclidean domain for all charts $\varphi, \psi$ of $N_1, N_2$ respectively.

    Finally, we call a pair $(N,h)$ a Gevrey-$\sigma$ manifold if $N$ and $h$ are Gevrey-$\sigma$.
\end{defin}
Of course, if $\sigma=1$, then a $G^1$ manifold is merely an analytic manifold.

That our definition of the space $G^\sigma(N)$ aligns with the one in \cite{KRS21} (given there only if $N$ is analytic) follows from 
\begin{lemma}\label{lem:GsigmaonManifold}
    Let $\sigma \geq 1$ and $(N,h)$ be a closed $G^\sigma$ manifold. We have $u \in G^\sigma(N)$ if and only if for some $C, R>0$
    \begin{equation}\label{eq:laplacerepeat}
        \norm{(-\Delta_h)^t u}_{L^2(N)} \leq C R^{2t}(2t)^{2t\sigma}\,, \quad t\geq 0\,.
    \end{equation}
    Furthermore, \eqref{eq:laplacerepeat} is true if and only if it holds for all integers $t\geq 0$, and some $C', R'>0$.
\end{lemma}
The proof is obtained by working in local coordinates and mimicking the proof of \cite{KRS21}*{Proposition~B.1} or relying on the Kotake-Narasimhan theorem \cite{fd2}*{Theorem~2.14}. 

In particular, we have as a consequence of (the proof of) \cite{KRS21}*{Proposition~B.1} that
\begin{lemma}\label{lem:GtoA}
    Let $\sigma \geq 1$ and $(N,h)$ be a closed $G^\sigma$ manifold. If $u\in L^2(N)$ satisfies \eqref{eq:laplacerepeat} for some $C, R >0$, then there is some $\rho_0>0$ so that for all $\rho \leq \rho_0$ we have $u\in A^{\sigma,\rho}(N)$. Conversely, if $u \in A^{\sigma,\rho}(N)$ for some $\rho > 0$, then $u\in G^\sigma(N)$. 
\end{lemma}

In this sense, if $(M,g)$ is a closed $G^\sigma$ manifold, then 
\[
G^{\sigma}(M)=\bigcup_{\rho>0}A^{\sigma,\rho}(M).
\]
Finally, we also define localized versions of these spaces. 

\begin{defin} \label{defin:localized_spaces}
Let $(M,g)$ be a Riemannian manifold, and let $S \subset M^{\mathrm{int}}$ be compact. Let $(M_{1},g_1)$ be a fixed closed manifold containing a neighborhood of $S$ on which $g_1$ agrees with $g$, let $s\in \mathbb{R}, \sigma\geq 1$ and $\rho>0$. We define
\begin{align*}
    H_{S}^{s}(M)&:=\{u \in H^{s} (M_{1}) : \supp(u) \subset S\}, \\
    A_{S}^{\sigma, \rho}(M)&:=\{u \in A^{\sigma, \rho}(M_{1}): \supp(u) \subset S\}.
\end{align*}
When $(M,g)$ is a $G^\sigma$ manifold, we will assume that $(M_1,g_1)$ is $G^\sigma$ as well.
\end{defin}

Our instability results are a consequence of the following general theorem.
\begin{thm}[\cite{KRS21}*{Theorem 6.2}] \label{thm:general_insteability_localized}
Let $M, N$ be smooth manifolds, let $S \subset M$ and $Q \subset N$ be compact, let $s, t \in \R$, and let $\delta>0$. Let $K_r=\{f \in H_{S}^{s+\delta}(M) ;\|f\|_{H^{s+\delta}} \le r\}$, and let $F$ be a map $K_{r} \to H_{Q}^{t}(N)$. Suppose that $\omega$ is a modulus of continuity such that
\[
\|f_{1}-f_{2}\|_{H^s(M)} \leq \omega (\|F(f_{1})-F(f_{2})\|_{H^{t}(N)}), \quad f_{1}, f_{2} \in K_{r}.
\]
\begin{enumerate}[label=(\alph*)]
    \item If there is $P \in \Psi_{\mathrm{cl}}^{0}(M)$, noncharacteristic at some point of $T^{*}S^{\mathrm{int}} \setminus\{0\}$, and $r_{0}>0$ so that $P(K_{r_{0}}) \subset K_{r}$ and $F \circ P$ maps $K_{r_{0}}$ into a bounded set of $H_{Q}^{t+m}(N)$ for any $m>0$, then for any $\alpha \in (0,1)$ one has $\omega(t) \gtrsim t^\alpha$ for $t$ small.
    \item If there is $P \in \Psi_{\mathrm{cl}}^{0}(M)$, noncharacteristic at some point of $T^{*}S^{\mathrm{int}} \setminus\{0\}$, and $r_0>0$ so that $P(K_{r_0}) \subset K_r$ and $F \circ P$ maps $K_{r_{0}}$ into a bounded set of $A_{Q}^{\sigma, \rho}(N)$ for some $1 \leq \sigma<\infty$ and $\rho>0$, then $\omega(t) \gtrsim|\log t|^{-\frac{\delta(\sigma \dim(N)+1)}{\dim(M)}}$ for $t$ small.
\end{enumerate}
\end{thm}

\section{Instability for the Minkowski case} \label{section:minkowski}

In this section we prove Theorem \ref{thm:instability}. To do so, we need two ingredients: control of the wavefront set, and mapping properties. Let us begin with the first one. To this end, we follow the setting of \cite{LOSU20}*{\S~2.3}, so we consider $L_{\kappa}$ as defined in \eqref{eq:lightray} as an FIO with canonical relation $C \subset T^{*}(\mathcal{M} \times M) \setminus 0$  given by 
\begin{equation} \label{eq:cr_minkowski}
    C=\{(z,\theta,\zeta,\hat\theta;t,x,\tau,\xi)\colon x=z+t\theta, \xi=\zeta,\tau=\theta\cdot\xi,\hat\theta = t(-\xi+(\xi\cdot\theta)\theta)\}. 
\end{equation}
Here, as in the introduction, $\mathcal{M}=\R_{z}^{n} \times S_{\theta}^{n-1}$, and $M=\R^{1+n}$.

We recall that the spacelike and timelike cones in $T^{*}\R^{1+n} \setminus 0$ are given by
\[
\Sigma_{s}=\{(t,x;\tau,\xi):|\tau|<|\xi|\}, \quad \Sigma_{t}=\{(t,x;\tau,\xi):|\tau|>|\xi|\},
\]
respectively. Let $\pi \colon C \to T^{*}\mathcal{M}$ be the projection, where $C$ is as in \eqref{eq:cr_minkowski}. If $(t,x,\tau,\xi) \in \pi (C)$, then for some lightlike geodesic $\gamma_{z,\theta}$ through $(t,x)$ we have $0=\langle (\tau,\xi),\dot{\gamma}_{z,\theta}(s) \rangle=g((\tau,\xi)^{\sharp},\dot{\gamma}_{z,\theta}(s))$. Since $\gamma_{z,\theta}$ is lightlike, we have two options for $(\tau,\xi)^{\sharp}$: it is spacelike, or null and parallel to $\dot{\gamma}_{z,\theta}$. In particular, $\pi(C) \cap \Sigma_{t}=\emptyset$. Then \cite{HormanderIV}*{Equation (25.2.2)} implies $\wf(L_\kappa f) \subset C \circ \wf(f)$. In particular, we obtain:

\begin{lemma} \label{lemma:wf_minkowski}
    If $\wf(f) \subset \Sigma_{t}$, then $\wf(L_{\kappa}f)=\emptyset$.
\end{lemma}

In the ultra-differentiable case, we have a similar result. For any $\sigma \in [1,\infty)$ we denote by $\wf_\sigma$ the Gevrey-$\sigma$ wavefront set, defined for example in \cite{HormanderI}*{\S8}. 

To obtain the analog of Lemma \ref{lemma:wf_minkowski}, we rely this time on \cite{KRS21}*{Proposition~D.3} (see also \cite{gramchev}*{Proposition~4.1}, \cites{tanig, lasc}). This allows us to state without proof that for any $\sigma>1$, $\kappa \in G^\sigma(\R^{1+n} \times S^{n-1})$, 
\[
    \wf_\sigma(L_\kappa) \subset \{(z,\theta,\zeta,\hat\theta;t,x,\tau,\xi)\colon x=z+t\theta, \xi=\zeta,\tau=\theta\cdot\xi,\hat\theta = t(-\xi+(\xi\cdot\theta)\theta)\}\,.
\]
In particular, $\wf_\sigma(L_\kappa)$ contains no nonzero timelike covectors. This observation combined with the fact that $(z,\theta,\zeta,\hat\theta;t,x,\tau,\xi) \in \wf_\sigma(L_\kappa)$ with $\zeta = 0 = \hat\theta$ if and only if $\tau = 0 = \xi$ and \cite{HormanderI}*{Theorem.~8.5.5} implies that
\begin{lemma}\label{lem:wfsigmaR} For $\sigma >1$, $\kappa \in G^\sigma(\R^{1+n} \times S^{n-1})$ and $f\in \mathcal{E}'(\R^{1+n})$,
\[ \wf_\sigma(f) \subset \Sigma_t \implies L_\kappa f \in G^\sigma(\R^{n}\times S^{n-1})\,.\]
\end{lemma}
Having settled the control of the wavefront set, we turn to obtaining the required mapping property. 
\begin{lemma} \label{lem:Lmap} 
Let $R$, $S$ and $Q$ be as in \eqref{eq:compacts}. We have
\[
L_{\kappa} \colon H_{S}^{0}(\R^{1+n}) \to  H_{Q}^{0}(\R^n\times S^{n-1}).
\]
\end{lemma}

\begin{proof}
By density, it is enough to consider $f \in C_{0}^{\infty}(\R^{1+n})$ with $\supp f \subset S$. For fixed $(z,\theta)$, the integrand is nonzero only for those $s$ such that $(s,z+s\theta)\in S$. In particular, since $S\subset [-R,R]\times \R^{n}$, the relevant set of $s$ has length at most $2R$. Thus, by Cauchy--Schwarz inequality,
\[
|L_\kappa f(z,\theta)|^2 \leq 2R\|\kappa\|_{L^{\infty}(S\times S^{n-1})}^2 \int_{\R}|f(s,z+s\theta)|^{2}ds.
\]
Integrating this over $(z,\theta) \in \R^{n} \times S^{n-1}$, we obtain 
\[
\|L_\kappa f\|_{L^{2}(\R^n\times S^{n-1})}^{2} \leq 2R\|\kappa\|_{L^{\infty}(S\times S^{n-1})}^{2} \int_{S^{n-1}}\int_{\R^n}\int_{\R}|f(s,z+s\theta)|^{2} ds dz d\theta.
\]
Now use the change of variable $y=z+s\theta$ to get
\[
\|L_{\kappa} f\|_{L^{2}(\R^{n}\times S^{n-1})}^{2} \leq 2R|S^{n-1}|\|\kappa\|_{L^{\infty}(S\times S^{n-1})}^2 \|f\|_{L^2(\R^{1+n})}^{2}.
\]
To show the support condition, we argue as follows. In the first place, note that if $L_{\kappa} f(z,\theta)\neq 0$, then there exists $s \in \R$ such that $(s,z+s\theta)\in S$, i.e., 
\begin{equation} \label{eq:ineq_support}
    s^{2}+|z+s\theta|^{2}\leq R^{2}.
\end{equation}
Now observe that the minimum (in $s$) of $s^{2}+|z+s\theta|^{2}$ is attained at $s=-(z\cdot\theta)/2$ and it is equal to $|z|^{2}-(z\cdot\theta)^{2}/2$. Since $\theta \in S^{n-1}$, we have $(z\cdot\theta)^{2} \leq |z|^2$. Thus, the minimum is bounded below by $|z|^{2}/2$. Therefore, \eqref{eq:ineq_support} holds only if $|z| \leq \sqrt{2}R$. Thus $\supp L_{\kappa} f\subset Q$, finishing the proof.    
\end{proof}

While Lemmas~\ref{lemma:wf_minkowski} and~\ref{lem:Lmap} suffice to prove Theorem~\ref{thm:instability} in the smooth setting, we require an additional ingredient for the Gevrey version. The following result is a minor extension of \cite{BLOS26}*{Proposition~6}.

\begin{lemma}\label{lem:upgradetoA}
    Let $\sigma\geq 1$, and $M$ be a smooth manifold and $N$ a $G^\sigma$ manifold and let $S \subset N^{\mathrm{int}}$ be compact. Let $F \colon L^2(M) \to G^\sigma(N)$ be linear with $F \colon L^2(M) \to L^2(N)$ continuous and $\supp F(L^2(M)) \subset S$. There is $\rho>0$ so that $F \colon L^2(M) \to A^{\sigma,\rho}_S(N)$ is continuous. 
\end{lemma}
Note that the statement is not vacuous for $\sigma=1$: the case of $N$ being a closed manifold and $S=N$ is non-trivial. 
\begin{proof}
    Let $(N_1,h_1)$ be a closed $G^\sigma$ manifold containing an open neighborhood of $S$ on which $h_1$ agrees with $h$. By extending $F$ as zero outside of $S$, and replacing $N$ by $N_1$, we may assume that $N$ is closed. 

    For integers $j,k \geq 1$ define the set
    \[
        A_{j,k} = \{u\in L^2(M)\colon \norm{Fu}_{A^{\sigma,1/j}(N)}\leq k\}\,,
    \]
    and note that $L^2(M) = \bigcup_{j, k =1}^\infty A_{j,k}$ because $G^\sigma(N) = \bigcup_{\rho>0}A^{\sigma,\rho}(N)$ by Lemma~\ref{lem:GtoA}. Each of the sets $A_{j,k}$ is closed: if $u_m \to u$ in $L^2(M)$ with $u_m \in A_{j,k}$ for all $m$, then $Fu_m \to Fu$ in $L^2(N)$ and so $\langle Fu_m,\varphi_l\rangle \to \langle Fu,\varphi_l\rangle$ for all eigenfunctions $\varphi_l$, so that Fatou's lemma implies 
    \[
        \norm{Fu}^2_{A^{\sigma,1/j}(N)} \leq \liminf_{m\to\infty} \norm{Fu_m}^2_{A^{\sigma,1/j}(N)} \leq k^2\,.
    \]
    The closedness of each $A_{j,k}$ and the Baire category theorem applied to $L^2(M)$ imply that there are some $j_\ast, k_\ast$ so that $A_{j_\ast,k_\ast}$ has non-empty interior, which is to say that there is some $u_0 \in L^2(M)$ and some $r>0$ so that $\overline{B^{L^2(M)}_r(u_0)} \subset A_{j_\ast,k_\ast}$. Thus, for any $u\in L^2(M)$ with $\norm{u} \leq 1$ we have
    \begin{align*}
        r\norm{Fu}_{A^{\sigma,1/j_\ast}(N)} \leq \norm{F(u_0+ru)}_{A^{\sigma,1/j_\ast}(N)} + \norm{Fu_0}_{A^{\sigma,1/j_\ast}(N)} \leq 2 k_\ast\,,
    \end{align*}
    which completes the proof.
\end{proof}

We finish this section by giving the proof of our instability result in the Minkowski case.

\begin{proof}[Proof of Theorem \ref{thm:instability}] \hfill
    \begin{enumerate}
        \item Choose a point $x_{0}=(t_{0},y_{0})\in S^{\mathrm{int}}$ and a timelike covector $\xi_{0}=(\tau_{0},\eta_{0}) \in T_{x_{0}}^{*}M\setminus 0$. Let $\psi \in C_{0}^{\infty}(S^{\mathrm{int}})$ satisfy $\psi(x_{0})=1$, and let $ \chi \in S_{\mathrm{cl}}^{0}(T^{*}M)$ be supported in a sufficiently small conic neighborhood of $(x_{0},\xi_{0})$, contained in the timelike cone $\Sigma_{t}$, with $\chi(x_{0},\xi_{0})\neq 0$. Define $P=\varepsilon\,\psi \Op(\chi)\psi$, where $\eps>0$ will be chosen small. Then
        \begin{itemize}
            \item $P\in \Psi_{\mathrm{cl}}^{0}(M)$,
            \item $P$ is noncharacteristic at $(x_{0},\xi_{0})$,
            \item $\supp(Pf)\subset \supp\psi\subset S$ for every distribution $f$.
        \end{itemize}
        Since $P$ has order zero, it is bounded on $H^{\delta}(M)$. Hence, by taking $\eps>0$ small enough, we can find $r_{0}>0$ such that $P(K_{r_{0}})\subset K_{r}$. In addition, since $\wf(P)\subset \Sigma_{t}$, by Lemma \ref{lemma:wf_minkowski}, we obtain $C\circ \wf(P)=\emptyset$. Hence, $L_{\kappa}\circ P$ is smoothing. Thus, for every $m>0$,
        \[
        L_{\kappa}\circ P \colon K_{r_{0}} \to H_{Q}^{m}(\R^{n}\times S^{n-1})
        \]
        has bounded image. The proof is finished after an application of Theorem \ref{thm:general_insteability_localized}(a).    
        \item We shall modify the choice of the pseudodifferential operator $P$ from above as follows: we assume in addition that $\psi, \chi \in G^\sigma$ (possible by \cite{HormanderI}*{Thm.1.4.10}, \cite{Luigi}*{\S~1.4})), and for some $C>0$ and all $\alpha$ satisfies
        \[ 
           \abs{\partial^\beta_{x}\partial^\alpha_\xi \chi} \leq C^{1+\abs{\alpha}+\abs{\beta}}(\alpha!\beta!)^\sigma(1+\abs{\xi})^{-\abs{\alpha}}\,,\quad \abs{\xi}>C\,,
        \]
        see \cite{Luigi}*{\S~3.4}. Thus, according to \cite{HormanderI}*{\S~8.5} or \cite{Luigi}*{\S~3.4}, we have
        \[
          \wf_\sigma(Pf) \subset \Sigma_t  \quad \forall f\in \mathcal{E}'(\R^{1+n})\,,
        \]
        and recall that there is some $r_0>0$ so that $P(K_{r_0}) \subset K_r$.
        
        According to Lemmas~\ref{lem:Lmap} and~\ref{lem:wfsigmaR}, we have
        \[
            L_\kappa \circ P \colon L^2(\R^{1+n}) \to G^\sigma( \R^n\times S^{n-1}),
        \]
        with $\supp L_\kappa \circ P(u) \subset Q$ for every $u\in L^2(\R^{1+n})$. Thus, by Lemma~\ref{lem:upgradetoA} there is some $\rho >0$ so that
        \[
            L_\kappa \circ P \colon K_{r_0} \subset L^2(\R^{1+n}) \to A^{\sigma,\rho}_{Q}(\R^n\times S^{n-1})
        \]
        is a continuous and linear map. In particular, the map $L_\kappa$ satisfies the assumptions of Theorem~\ref{thm:general_insteability_localized} and $L_\kappa\circ P$ part (b) thereof, which completes the proof.
    \end{enumerate}
\end{proof}

\section{Instability on smooth manifolds} \label{section:smooth}

Let $(M,g)$ be a Lorentzian manifold of dimension $n+1$. We will follow \cite{LOSU20}*{Section 3}. We work on the space of lightlike geodesics close to a fixed lightlike geodesic $\gamma_{0} \colon [0,\ell] \to M$. To give a smooth structure to this set, we argue as follows. Since we work near $\gamma_{0}$, we can parameterize lightlike geodesics near $\gamma_{0}$ by choosing a spacelike hypersurface $H$ containing $\gamma_{0}(0)$ and semigeodesic coordinates associated to $H$, $(t,z) \in (-T,T) \times Z \subset \R \times \R^{n}$. In these coordinates, $H=\{t=0\}$ and $g=-dt^{2}+g'$, with $g'=g'(t, z)$ a Riemannian metric on $Z$ that depends smoothly on $t$. Moreover, the coordinates are chosen so that $\gamma_{0}(0)=0$ and $\dot{\gamma}_{0}(0)=(1, \theta_{0})$ where, writing $h=g'(0, \cdot)$, we have $(\theta_{0}, \theta_{0})_h=1$ since $\gamma_{0}$ is lightlike. On the unit sphere bundle $SZ$ (with respect to $h$), we choose local coordinates of the form $(z,a) \in Z \times A$, (where $A \subset \R^{n-1}$), so that $\theta(0,0)=\theta_{0}$, where 
\begin{align*}
    Z \times A &\to SZ, \\
    (z,a) &\mapsto (z,\theta(z,a)).
\end{align*}
We denote $\gamma_{z, a}$ the geodesic that at zero passes through $(0, z)$ with velocity $(1, \theta(z, a))$. As is customary in the literature, we denote $\mathcal{M}=Z \times A$. 

Let $\Omega \subset M$ be open and relatively compact, and suppose that the end points $\gamma_{0}(0)$ and $\gamma_{0}(\ell)$ are outside $\overline{\Omega}$. By making $Z$ and $A$ smaller, we may assume without loss of generality that the end points $\gamma_{z, a}(0)$ and $\gamma_{z, a}(\ell)$ are outside $\bar{\Omega}$ for all $(z, a) \in \overline{\mathcal{M}}$. Throughout this section, we consider the local version of the light ray transform defined as follows
\begin{equation}\label{eq:Llorentz}
L_{\kappa} f(\gamma)=\int_{0}^{\ell} \kappa(\gamma(s), \dot{\gamma}(s)) f(\gamma(s)) ds, 
\end{equation}
where $f \in C_{0}^{\infty}(\Omega)$, $\gamma=\gamma_{z, a}$, $(z, a) \in \mathcal{M}$, and $\kappa \in C^{\infty}$ is a weight function, positively homogeneous of degree $1$ in the second variable (so that it is independent of the parametrization). 

Since the following considerations are local, we restrict, if necessary, to a neighborhood in $\mathcal{M} \times (0,\ell)$ on which the map $(z,a,s) \mapsto (z,a,\gamma_{z,a}(s))$ is injective. We denote its image in $\mathcal{M} \times \Omega$ by $X$. The point-geodesic relation
\[
X=\{(z, a, x) \in \mathcal{M} \times \Omega : x=\gamma_{z, a}(s) \text { for some } s \in(0, \ell)\},
\]
is a smooth $2n$ dimensional submanifold of the $3n$ dimensional $\mathcal{M} \times \Omega$, parameterized by the map $(z, a, s) \mapsto (z, a, \gamma_{z, a}(s))$. Writing $x=\gamma_{z, a}(s)$, this map has differential
\[ 
\begin{pmatrix}
    \mathrm{Id} & 0 & 0 \\
0 & \mathrm{Id} & 0 \\
\partial x / \partial z & \partial x / \partial a & \dot{\gamma}_{z,a}(s)
\end{pmatrix} \begin{pmatrix}
    dz \\ da \\ ds
\end{pmatrix}.
\] 
which has maximal rank $2n$. Thus, after the above localization, this map is an embedding onto its image, and $X$ is a smooth $2n$-dimensional embedded submanifold of $\mathcal{M} \times \Omega$. The conormal bundle $N^{*}X$ at any point is the space conormal to the range of that differential. The canonical relation $C:=N^{*}X \subset T^{*}(\mathcal{M} \times \Omega) \setminus 0$ of $L_\kappa$ is given by 
\begin{equation} \label{eq:canonical_relation}
    \begin{split}
    C=\big\{ &\left( (z,a,\zeta,\alpha),(x,\xi) \right) : x=\gamma_{z,a}(s), \, \langle \xi,\dot{\gamma}_{z,a}(s)\rangle=0, \, \zeta_{j}=\langle \xi,\partial_{z^{j}}\gamma_{z,a}(s)\rangle,  \\
    &   \text{for } j=1,\ldots,n, \, \alpha_{k}=\langle \xi,\partial_{a^{k}}\gamma_{z,a}(s)\rangle \text{ for } k=1,\ldots,n-1, \, s\in (0,\ell) \big\},
\end{split}
\end{equation}
We see from the calculation of the differential of $(z,a,s)\mapsto (z,a,\gamma_{z,a}(s))$ as well as \cite{LOSU20}*{Lemma~3.1} that
\begin{lemma}\label{lem:sumbersions}
    The projections $\pi_{\mathcal{M}}\colon X\to \mathcal{M}$ and $\pi_{M}\colon X\to M$ are submersions. 
\end{lemma}
Let
\[
\Sigma_{s}=\{(x,\xi) \in T^{*}M:|\xi|_{g}>0\}, \quad \Sigma_{t}=\{(x,\xi) \in T^{*}M:|\xi|_{g}<0\},
\]
denote the spacelike and timelike cone on the phase space, respectively. Let $\pi \colon C \to T^{*}\Omega$ be the projection, where $C$ is as in \eqref{eq:canonical_relation}. If $(x,\xi) \in \pi (C)$, then for some lightlike geodesic $\gamma_{z,a}$ through $x$ we have $0=\langle \xi,\dot{\gamma}_{z,a}(s) \rangle=g(\xi^{\sharp},\dot{\gamma}_{z,a}(s))$. Since $\gamma_{z,a}$ is lightlike, we conclude that $\xi^{\sharp}$ is spacelike or null and parallel to $\dot{\gamma}_{z,a}$. In particular, $\pi(C) \cap \Sigma_{t}=\emptyset$. By wavefront set calculus for Fourier integral operators, we have $\wf(L_\kappa f) \subset C \circ \wf(f)$, see \cite{HormanderIV}*{Equation (25.2.2)}. We conclude that 
\begin{lemma} \label{lemma:wf}
    If $\wf(f) \subset \Sigma_{t}$, then $\wf(L_{\kappa}f)=\emptyset$.
\end{lemma}

This gives the regularity property that will be needed for the instability result. However, we also need a mapping property. The first one follows from the Sobolev embedding.

\begin{lemma}\label{lem:L^2cptmap}
Let $S\Subset \Omega$ be compact, and define
\[
Q_{S}=\{(z,a)\in \overline{\mathcal{M}}: \gamma_{z,a}([0,\ell])\cap S\neq \emptyset\}.
\]
Then $Q_{S}$ is compact, and we have the following mapping property 
\[
L_{\kappa} \colon H^{0}_{S}(\Omega)\to H^{0}_{Q_{S}}(\mathcal{M}).
\]
\end{lemma}

\begin{proof}
We first prove that $Q_{S}$ is compact. Let
\[
E_{S}=\{(z,a,s)\in\overline{\mathcal{M}}\times[0,\ell]:
 \gamma_{z,a}(s)\in S\}.
\]
Since $\Phi$ is continuous and $S$ is compact, $E_{s}$ is a closed subset of a compact set, hence it is compact itself. The compactness of $Q_{S}$ now follows from the fact that it is the projection of $E_{S}$ under $\overline{\mathcal{M}} \times [0,\ell] \to \overline{\mathcal{M}}$. 

Now we deal with the mapping property.  First, since $\overline{\mathcal{M}}\times [0,\ell]$ is compact and $\kappa$ is smooth, we have
\[
C_{\kappa}:=\sup_{\substack{(z,a)\in \overline{\mathcal{M}}\\ s\in[0,\ell]}}|\kappa(\gamma_{z,a}(s),\dot{\gamma}_{z,a}(s))|<\infty.
\]
Let $f \in H_{S}^{0}(\Omega)$. The Cauchy-Schwarz inequality gives
\begin{equation} \label{eq:ineqCS}
    |L_{\kappa} f(z,a)|^2 \leq \ell C_
{\kappa}^{2} \int_{0}^{\ell}|f(\gamma_{z,a}(s))|^{2} ds.
\end{equation}
Note that by \cite{LOSU20}*{Lemma 3.1}, the evaluation map
\[
\Phi \colon \overline{\mathcal{M}}\times[0,\ell]\to M, \qquad \Phi(z,a,s)=\gamma_{z,a}(s),
\]
is a submersion on \(\Phi^{-1}(S)\). Indeed, that $\Phi$ is a submersion means that $d\Phi(z,a,s) \colon T_{(z,a,s)}(\mathcal{M}\times[0,\ell]) \to T_{\gamma_{z,a}(s)}M$ has full rank $n+1$ whenever $\gamma_{z,a}(s)\in S$. Equivalently,
\[
\mathrm{span}\{\dot{\gamma}_{z,a}(s), \partial_{z^{1}}\gamma_{z,a}(s),\ldots,\partial_{z^{n}}\gamma_{z,a}(s),\partial_{a^{1}}\gamma_{z,a}(s),\ldots, \partial_{a^{n-1}}\gamma_{z,a}(s)\} =T_{\gamma_{z,a}(s)}M.
\]
Thus the condition says that these Jacobi fields $\partial_{z^{j}}\gamma_{z,a}(s)$, $\partial_{a^{k}}\gamma_{z,a}(s)$, together with $\dot\gamma_{z,a}$, span the whole tangent space along the part of the geodesic meeting $S$, which is exactly the conclusion of the cited result. 

Now observe that since $\Phi^{-1}(S)$ is compact, and the submersion property is open, there exists an open neighborhood $U$ of $\Phi^{-1}(S)$ such that $\Phi$ is a submersion on $U$. Take $\chi\in C_{0}^{\infty}(U)$ such that $\chi=1$ on a neighborhood of $\Phi^{-1}(S)$. Since $\supp f\subset S$, we have $|f(\Phi(z,a,s))|^{2} = \chi(z,a,s)|f(\Phi(z,a,s))|^{2}$.

Then,
\[
\begin{split}
    \|L_{\kappa}f\|_{L^2(\mathcal{M})}^2 &\leq \ell C_{\kappa}^{2}\int_{\mathcal{M}}\int_{0}^{\ell} |f(\gamma_{z,a}(s))|^{2} dsdzda \\
    &=\ell C_{\kappa}^{2}\int_{\mathcal M}\int_{0}^{\ell} |f(\Phi(z,a,s))|^{2} dsdzda \\
    &=\ell C_{\kappa}^{2}\int_{\mathcal M}\int_{0}^{\ell} \chi (z,a,s) |f(\Phi(z,a,s))|^{2} dsdzda \\
    &=\ell C_{\kappa}^{2}\int_{S} |f(x)|^2 \rho(x) dV(x),
\end{split}
\]
where the first inequality follows from integrating \eqref{eq:ineqCS} over $\mathcal{M}$, on the second line we used the definition of $\Phi$, on the third line we used the fact that since $\supp(f)\subset S$ (then the integrand is supported in $\Phi^{-1}(S)$), and in the last step we used the fact that $\Phi$ is a submersion on $\Phi^{-1}(S)$ and we applied the co-area formula (see for instance \cite{EG25}*{\S 3.4.3}),
where
\[
\rho(x)=\int_{\Phi^{-1}(x)} \frac{\chi(z,a,s)}{J_\Phi(z,a,s)}d\sigma_{x}(z,a,s).
\]
Here $J_{\Phi}$ is the normal Jacobian of $\Phi$, and $d\sigma_{x}$ is the induced measure on the fiber $\Phi^{-1}(x)$.

Since $\Phi^{-1}(S)$ is compact and $\Phi$ is a submersion there, $J_{\Phi}$ is bounded away from zero on $\Phi^{-1}(S)$. Furthermore, by definition of $\chi$, this function is constant over there. Moreover, the fibers $\Phi^{-1}(x)$, $x \in S$, have uniformly bounded $(n-1)$-dimensional volume since $\Phi$ is continuous and $S$ is compact. Hence $\rho$ is bounded on $S$. Therefore,
there exists $C>0$ such that $\rho(x) \leq C$ for any $x \in S$. Combining this with the previous inequalities, we obtain that
\[ 
\|L_{\kappa}f\|_{L^{2}(\mathcal{M})}^{2} \leq \ell C_{\kappa}^{2}C \|f\|_{L^{2}(\Omega)}^{2}.
\]
The support statement follows from the definitions.    
\end{proof}

Finally, we are able to state and prove the instability result, which is the analog of Theorem \ref{thm:instability} (a), for smooth manifolds.

\begin{thm}\label{thm:manCinftyinsta}
Let $S\Subset \Omega$ be compact with nonempty interior, and let
\[
Q_S=\{(z,a)\in \overline{\mathcal{M}} : \gamma_{z,a}([0,\ell])\cap S\neq \emptyset\}.
\]
Let $\delta,r>0$, and $K_{r}=\{ f \in H_{S}^{\delta}(\Omega) : \|f\|_{H^{\delta}(\Omega)}\leq r\}$. Suppose that $\omega$ is a modulus of continuity such that
\[
\|f_{1}-f_{2}\|_{L^{2}(\Omega)} \leq \omega (\|L_{\kappa}f_{1}-L_{\kappa} f_{2}\|_{L^{2}(\mathcal{M})}), \qquad f_{1},f_{2}\in K_{r}.
\]
For every $\alpha \in (0,1)$, $\omega(t) \gtrsim t^\alpha$ for small $t$.
\end{thm}

\begin{proof}
Choose a point $x_{0} \in S^{\mathrm{int}}$ and let $\xi_{0}$ be a timelike co-vector so that $(x_{0},\xi_{0}) \in T^{*}\Omega \setminus \{0\}$. Let $\psi \in C_{0}^{\infty}(S^{\mathrm{int}})$ satisfy $\psi(x_{0})=1$, and let $\chi \in S_{\mathrm{cl}}^0(T^*\Omega)$ be a classical symbol supported in a sufficiently small conic neighborhood of $(x_{0},\xi_{0})$, contained in the timelike cone $\Sigma_{t}$, with $\chi(x_{0},\xi_{0}) \neq 0$. Let $P=\eps \psi \Op(\chi)\psi$,
where $\eps>0$ will be chosen small. Then
\begin{enumerate}
    \item $P \in \Psi^0_{\mathrm{cl}}(\Omega)$,
    \item $P$ is noncharacteristic at $(x_{0},\xi_{0})$,
    \item $\supp(Pf)\subset \supp\psi\subset S$, for every distribution $f$.
\end{enumerate}
Since $P$ is bounded on $H^{\delta}(\Omega)$, choosing $\eps>0$ small enough we have $P(K_{r_{0}})\subset K_{r}$ for some $r_{0}>0$. Since $\wf(P) \subset \Sigma_{t}$, by Lemma \ref{lemma:wf} and Lemma~\ref{lem:L^2cptmap} we obtain that $C \circ \wf(P)=\emptyset$. Hence, $L_{\kappa} \circ P$ is smoothing. Thus, for every $m>0$,
\[
L_{\kappa} \circ  P \colon K_{r_{0}} \to H_{Q_{S}}^{m}(\mathcal{M})
\]
has bounded image. The result now follows from Theorem \ref{thm:general_insteability_localized} (a).    
\end{proof}

\begin{rmk}
By \cite{LOSU20}*{Theorem~3.1}, one can construct $\chi \in \Psi^{0}$ so that $N=L_{\kappa}^{*}\chi L_{\kappa} \in \Psi^{-1}$ with essential support in the spacelike cone, and if $\kappa$ is non-vanishing, one can even choose $\chi$ so that $N$ is elliptic for any $(x_{0},\xi_{0}) \in \Sigma_{s} \cap N^{*}\gamma_{0}$. This again gives the correct mapping properties between Sobolev spaces, which can be used to conclude an instability result for $N$ using Theorem \ref{thm:general_insteability_localized} (a). Observe, however, that one cannot get a mapping property of the type $\chi^{1/2}L_{\kappa} \colon H^{s}(\Omega) \to H^{s+1/2}(\mathcal{M})$ from the ones of $L_{\kappa}^{*}\chi L_{\kappa}$. Indeed, this would require mapping properties for $L_{\kappa}^{*}\chi (1-\Delta_{\mathcal{M}})^{s+1/2}L_{\kappa}$, which is a priori not a pseudodifferential operator.
\end{rmk}

\section{Instability in the Gevrey category} \label{section:gevrey}

Fix $\sigma > 1$ and $\kappa \in G^\sigma(T M\setminus 0)$ positively homogeneous of degree $1$ in the fiber variable and assume now that the Lorentzian manifold $(M,g)$ is $G^\sigma$. Let $\Sigma_t$ and $X$ be as in the previous section. According to \cite{MST23}, \cite{LOSU20} (in particular Lemma~\ref{lem:sumbersions}) the light ray transform \eqref{eq:Llorentz} satisfies the assumptions of Theorem~\ref{thm:doublefibFIO} (it is a double fibration transform). Thus, using the oscillatory integral representation of $L_\kappa$, as well as \cite{KRS21}*{Proposition~D.3} we find as before that
\[
    \wf_\sigma(L_\kappa) \subset N^\ast X\setminus 0\,.
\]
An application of Lemma~\ref{lem:sumbersions} and \cite{MST23}*{Lemma~5.3} together with \cite{HormanderI}*{Theorem~8.5.5} imply that the same discussion leading to Lemma~\ref{lemma:wf} gives
\begin{lemma} \label{lemma:wfsigma}
    If $\wf_\sigma(f) \subset \Sigma_{t}$, then $\wf_\sigma(L_{\kappa}f)=\emptyset$.
\end{lemma}

We can now state
\begin{thm} \label{thm:gevrey-insta}
Let $\sigma>1$ and $S\Subset \Omega$ be compact with nonempty interior, and let
\[
Q_S=\{(z,a)\in \overline{\mathcal{M}} : \gamma_{z,a}([0,\ell])\cap S\neq \emptyset\}.
\]
Let $\delta,r>0$, and $K_{r}=\{ f \in H_{S}^{\delta}(\Omega) : \|f\|_{H^{\delta}(\Omega)}\leq r\}$. If $\omega$ is a modulus of continuity such that
\[
\|f_{1}-f_{2}\|_{L^{2}(\Omega)} \leq \omega (\|L_{\kappa}f_{1}-L_{\kappa} f_{2}\|_{L^{2}(\mathcal{M})}), \qquad f_{1},f_{2}\in K_{r}\,,
\]
then $\omega(t) \gtrsim \abs{\log t}^{-\frac{\delta((2n-1)\sigma+1)}{n+1}}$ for $t$ small.
\end{thm}
The proof is a copy of the proof of Theorem~\ref{thm:manCinftyinsta}, wherein we must only modify the operator $P$ to match the Gevrey-$\sigma$ setting as we did in the proof of Theorem~\ref{thm:instability}. We  rely on Lemmas~\ref{lemma:wfsigma},~\ref{lem:L^2cptmap},~\ref{lem:upgradetoA} to obtain that there are some $r_0,\rho>0$ so that
\[
    L_\kappa \circ P\colon K_{r_0} \to A^{\sigma,\rho}_{Q_S}(\mathcal{M})
\]
is continuous, which allows the application of Theorem~\ref{thm:general_insteability_localized} (b).

\section{Stability of moments} \label{section:stability}

The goal of this section is to prove Theorem~\ref{thm:moment-stability-intro}. To this end, we will begin by briefly recalling the geometries therein and their relation with magnetic and $\MP$-systems, defined in Appendix~\ref{appendix:magnetic}. We recall that such relations have been used in \cites{Stefanov24, MT25, OPS25}. Here, we principally follow the second article in that list.

Consider a standard stationary manifold $(M,g)$, that is, $M=\R \times N$, and $g$ is of the form
\begin{equation} \label{eq:metric2}
    g=-\lambda(x)(dt-\omega_{i} dx^{i})^{2} +h_{ij}(x)dx^{i}dx^{j},
\end{equation}
where $\lambda \in C^{\infty}(N)$ is a positive function, $\omega$ is a 1-form in $N$, and $h$ is a Riemannian metric in $N$. Here $\langle \bullet,\bullet \rangle$ is reserved to denote the pairing of 1-forms with vectors. We will denote by $(\bullet,\bullet)_{g}$ the inner product with respect to $g$. The condition of $\gamma=(t,x)$ being a lightlike geodesic reads
\[ 
-\lambda(x)(\dot{t}-\langle \omega,\dot{x}\rangle)^{2}+|\dot{x}|_{h}^{2}=0.
\]
It is known that $\partial_{t}$ is a Killing field for $g$, hence, for any geodesic we have $(\dot{\gamma},\partial_{t})_{g}=\rho$ for some $\rho \in \R$. Explicitly, this means that
\begin{equation} \label{eq:const2}
    -\lambda(x)(\dot{t}-\langle \omega,\dot{x} \rangle)=\rho.
\end{equation}
We will say that $\gamma$ has \emph{momentum} $\rho$, nomenclature that comes from symplectic geometry and classical mechanics. 

We are interested in lightlike geodesics pointing into $N$. It is then natural to study 
\[ 
\partial T^{0}M=\{(t,x,v_{t},v_{x}):(t,x) \in \partial M, \, |(v_{t},v_{x})|_{g}=0\}.
\]
However, we will require a more precise description. We consider lightlike directions in the boundary with momentum $\rho$:
\begin{align*}
    \partial T^{0,\rho}M  :&=\{(t,x,v_{t},v_{x}):(t,x) \in \partial M, \, |(v_{t},v_{x})|_{g}=0, \, ( \partial_{t},(v_{t},v_{x}) )_{g}\equiv \rho\} \\
    &=\{(t,x,\langle \omega,v_{x}\rangle-\rho/\lambda(x),v_{x}):(t,x) \in \partial M, \, |(v_{t},v_{x})|_{g}=0\}.
\end{align*}
Note that when $\rho=0$, $\partial T^{0,\rho}M$ is just the zero section, which gives rise to constant geodesics, so  from now on, we assume $\rho \neq 0$. Now we consider the $\MP$-system $(N,h,d(\rho \omega),\rho^{2}/(-2\lambda))$, see \S\ref{subsection:mp} for the definition. Given $k \in \R$, the $k$-energy bundle and the $k$-influx boundary take the form
\begin{align*}
    S_{\rho}^{k}N &:=S^{k}N=\left\{ (x,v_{x}) \in TN: \frac{1}{2}|v_{x}|_{h}^{2}+\frac{\rho^{2}}{-2\lambda}=k \right\}, \\
    \partial_{+}S_{\rho}^{k}N &:=\partial_{+}S^{k}N =\{(x,v_{x}) \in S^{k}N:x \in \partial N, \, (v_{x},\nu_{x}(x))_{h} \geq 0\} ,
\end{align*}
respectively, where $\nu_{h}(x)\in T_{x}N$ denotes the inward $h$-unit normal to $\partial N$. We also set $\nu_{g}=(\langle \omega_{x} ,\nu_{h}(x)\rangle,\nu_{h}(x))$, i.e., $\nu_{g}$ is the inward $g$-unit normal to $\partial M=\R \times \partial N$. With these notations, we see that $g((v_{t},v_{x}),\nu_{g})=h(v_{x},\nu_{h})$ and 
\[
\partial_{+}T^{0,\rho}M:=\{(t,x,v_{t},v_{x}) \in \partial T^{0,\rho}M:((v_{t},v_{x}),\nu_{g}(x))_{g} \geq 0\},
\]
is just the graph of $(t,x,v_{x}) \mapsto \langle \omega,v_{x}\rangle -\rho/\lambda(x)$ over $\R \times \partial_{+}S_{\rho}^{0}N$. Thus $\partial_{+}T^{0,\rho}M$ is naturally identified with $\R\times \partial_{+}S_\rho^0N$. In particular, each point $(T,X,V_{x})\in \R \times \partial_{+}S_{\rho}^{0}N$ determines a unique lightlike geodesic $\gamma=\gamma_{T,X,V_x}$ with momentum $\rho$ and initial conditions $(\gamma(0),\dot{\gamma}(0))=(T,X,\langle \omega,V_{x}\rangle -\rho/\lambda(x),V_{x})$. Furthermore, we have that $(\gamma(s),\dot{\gamma}(s))$ is given by 
\begin{equation} \label{eq:lightlike-geo-rho}
    \left( T+\int_{0}^{s}  \left( \langle\omega_{x(\sigma)},\dot{x}(\sigma) \rangle -\frac{\rho}{\lambda(x(\sigma))} \right) d\sigma,x(s),\langle \omega_{x(s)},\dot{x}(s)\rangle-\frac{\rho}{\lambda(x(s))},\dot{x}(s) \right),
\end{equation}
where $\dot{x}(0)=V_{x}$. Now we make the connection between the lightlike ray transform on $(M,g)$ and the $\MP$-ray transform on $(N,h,d(\rho \omega),\rho^{2}/(-2\lambda))$, i.e., the operator that integrates functions over $\MP$-geodesics, we refer to \S\ref{subsection:mp} for its definition. We begin by using \cite{FIO21}*{Equation~(16)}, which states that for any $f\in C_{0}^{\infty}(M)$ and $\tau \in \R$ we have
\begin{equation} \label{eq:lightray_fourier}
    \int_{\R}e^{-i \tau T}L_{\gamma}fdT=\int_{I}e^{i\tau t(s)}\mathcal{F}_{t}(f)(\tau,x(s))ds,
\end{equation}
where $\gamma=\gamma_{T,X,V_{x}}$ with $(T,X,V_{x}) \in \partial_{+}T^{0,\rho}M$, $I$ is the interval of definition of $\gamma$, and $t(s)$ is the second term in the first entry of \eqref{eq:lightlike-geo-rho}. In particular, taking $\tau=0$ we find
\[ 
\int_{\R} L_{\gamma}f dT=I_{\mathcal{MP}}([\mathcal{F}_{t}f](0,\bullet)),
\]
where the operator on the right-hand side is the $\mathcal{MP}$-ray transform, see \S\ref{subsection:mp}. Observe that the left-hand side can be thought of as the zero moment of $L_{\gamma}f$. Thus,
\[
|I_{\MP}[\mathcal{F}_{t}f](0,\bullet)| \leq \|L_{\gamma}f\|_{L^{1}(\R)},
\]
for any $(X,V_{x}) \in \partial_{+}S^{0}N$. Hence,
\[ 
\norm{I_{\MP}(\mathcal{F}_{t}f)(0,\bullet)}_{H^{1/2}(\partial_{+}S^{0}N)} \leq \norm{L_{\gamma}f}_{H_{X,V_{x}}^{1/2}L_{T}^{1}}.
\]
Furthermore, we can use the relation between $\MP$-systems and magnetic ones. Let $(N,H,d(\rho \omega))$ be the magnetic reduction of $(N,h,d(\rho \omega),\rho^{2}/(-2\lambda))$, defined in \S\ref{subsection:magnetic}. Let $I_{\mu}$ be the corresponding magnetic ray transform. By \cite{MT24}*{Proposition 3.7}, we have
\[ I_{\MP}f(X,V_{x})=I_{\mu}(f/(2k+\rho^{2}/\lambda))(X,V_{x}/(2k+\rho^{2}/\lambda)), \]
so that 
\[ 
\|I_{\mu}(\mathcal{F}_{t}f)(0,\bullet)\|_{H^{1/2}(\partial_{+}SN)} \leq C\norm{L_{\gamma}f}_{H_{X,V_{x}}^{1/2}L_{T}^{1}}.
\]
Therefore, to obtain a stability estimate for the light ray transform, it is enough to obtain a stability result for $I_{\mu}$. Before taking moments, $L_{\gamma_{T,X,V}}f$ is a function of $(T,X,V)\in \R \times \partial_{+}S_{\rho}^{0}N$, which is identified with $\partial_{+}T^{0,\rho}M$. Hence, after integration in $T$, we obtain functions over $\partial_{+}S_{\rho}^{0}N$.

In a more general way, \eqref{eq:lightray_fourier} gives
\begin{equation} \label{eq:lightray_fourier-derivatives}
    \int_{\R}T^{k}L_{\gamma}fdT=i^{k}\sum_{j=0}^{k} \binom{k}{j} \int_{I} (it(s))^{k-j}\partial_{\tau}^{j}\mathcal{F}_{t}f(0,x(s))ds.
\end{equation}

Hence the previous argument can be generalized to obtain the following stability result for the moments of the light ray transform.

\begin{thm} \label{thm:moment-stability}
Let $(N,h,d(\rho\omega),\rho^{2}/(-2\lambda))$ be $0$-simple. Then, for every $k \geq 0$, there exist constants $A_{k}=A_{k}(\rho)>0$ such that for any $f\in L^{1} (\R_t,\langle t\rangle^{k} dt;L^2(N))$ we have
\[
\left\| \int_{\R}t^{k}f(t,\bullet)dt \right\|_{L^{2}(N)} \leq A_{k} \sum_{r=0}^{k} \left\| \int_{\R} T^{r}(L_{\gamma_{T,X,V}}f)(T,\bullet)dT \right\|_{H^{1/2}(\partial_{+}S_{\rho}^{0}N)}.
\]
\end{thm}

Here, $\langle t\rangle=(1+t^{2})^{1/2}$ is the usual Japanese bracket. 

\begin{rmk}
Theorem \ref{thm:moment-stability-intro} is a simplified version of Theorem \ref{thm:moment-stability}. More precisely, in the full statement the relevant dynamics is the 0-simple $\MP$-system $(N,h,d(\rho\omega),\rho^{2}/(-2\lambda))$, and the constants depend on the corresponding energy level. On the other hand, Theorem \ref{thm:moment-stability-intro} is stated for the magnetic reduction of, say, the $\MP$-system with $\rho=1$, so that the dependence on the constants can be avoided. In the full version, the constants depend on the used $\MP$-system only, but not on other possible systems, i.e., we only care about the $\MP$-system for one fixed $\rho$. Moreover, Theorem \ref{thm:moment-stability} is stated on the natural incoming boundary bundle $\partial_{+}S_{\rho}^{0}N$, while in the introductory formulation this dependence is suppressed in the notation and written simply as $\partial_{+}SN$. Finally, the full theorem applies to the larger class $f\in L^{1}(\R_{t},\langle t\rangle^{k} dt;L^2(N))$, which is the natural condition ensuring that the $k$-th time moment of $f$ is well-defined as an $L^2(N)$-function. In particular, the introductory statement follows from the full theorem for $f\in C^\infty_{0}(M)$.   
\end{rmk}

Before giving the proof of Theorem \ref{thm:moment-stability}, we will need a continuity result for certain auxiliary operators. Let us define
\[
M_{k}(X,V)=\int_{\R}T^{k}L_{\gamma_{T,X,V}}fdT, \quad J_{\ell}(X,V)=\int_{I}t_{X,V}(s)^{\ell}f(x_{X,V}(s))ds,
\]
\[
m_{j}(x)=\partial_\tau^{j}\mathcal{F}_{t}f(0,x)=\left.\partial_\tau^{j}\left(\int_{\R} e^{-i \tau t} f(t, x) d t\right) \right|_{\tau=0}=\int_{\R}(-it)^j f(t, x) dt.
\]

\begin{lemma}\label{lemma:bound-aux}
Suppose that the $\MP$-system $(N,h,d(\rho \omega),\rho^{2}/(-2\lambda))$ is 0-simple. There are constants $C_{\ell}>0$ such that
    \[
    \|J_{\ell} u\|_{H^{1 / 2}(\partial_{+} S_{\rho}^{0} N)} \leq C_{\ell}\|u\|_{L^2(N)}, \quad \ell \geq 1
    \]
\end{lemma}

\begin{proof}
First, we claim that $w_{\ell}(s,X,V_{x}):=t_{X,V_{x}}(s)^{\ell}$ is smooth and uniformly bounded, together with all derivatives of finite order, on the relevant compact set of trajectories. Indeed, since $|\dot{x}|_{h}=\rho/\sqrt{\lambda}$, we have 
\[
|t_{X,V_{x}}(s)| \leq \left( \|\omega\|_{L^{\infty}(N)}|\rho|\|\lambda^{-1/2}\|_{L^{\infty}(N)}+|\rho|\|\lambda^{-1}\|_{L^{\infty}(N)} \right) s\,.
\]
Since the $\MP$-system is simple, it is non-trapping, hence $s\leq \tau(X,V_{x})\leq S_{\max}<\infty$. This gives a bound for $w_{\ell}$, while its derivatives can be bounded in a similar way. Hence, the result can be obtained in the same way as Lemma \ref{lemma:mapping-prop-magnetic}.   
\end{proof}

\begin{proof}[Proof of Theorem~\ref{thm:moment-stability}]
In the new notations, we have to prove that for each $k\geq 0$:
\[
\|m_{k}(x)\|_{L^{2}(N)} \leq A_{k} \sum_{r=0}^{k} \| M_{r} \|_{H^{1/2}(\partial_{+}S_{\rho}^{0}N)}.
\]
From \eqref{eq:lightray_fourier-derivatives}, we have
\[
I_{\MP}(m_{k})=i^{-k}M_{k}-\sum_{j=0}^{k-1} \binom{k}{j}i^{k-j}J_{k-j}(m_{j}).
\]
Hence,
\[
\|I_{\MP}m_{k}\|_{H^{1/2}(\partial_{+}S_{\rho}^{0}N)} \leq \|M_{k}\|_{H^{1/2}(\partial_{+}S_{\rho}^{0}N)}+\sum_{j=0}^{k-1}\binom{k}{j} \|J_{k-j}m_{j}\|_{H^{1/2}(\partial_{+}S_{\rho}^{0}N)}.
\]
Using the stability of the $\MP$-ray transform (Theorem \ref{thm:MP-stability}) on the left-hand side and Lemma \ref{lemma:bound-aux} on the right-hand side, we find
\begin{equation} \label{eq:ineq-induction}
    \|m_{k}\|_{L^{2}(N)} \leq C_{0}\|M_{k}\|_{H^{1/2}(\partial_{+}S_{\rho}^{0}N)} + C_{0} \sum_{j=0}^{k-1} \binom{k}{j}C_{k-j}\|m_{j}\|_{L^{2}(N)}.
\end{equation}
We now prove the theorem by induction over $k$. The case $k=0$ is given by the previous discussion, so let us assume that $k>0$ and that for every $j<k$, there exists $A_{j}>0$ such that 
\[
\|m_{j}\|_{L^{2}(N)} \leq A_{j}\sum_{r=0}^{j}\|M_{r}\|_{H^{1/2}(\partial_{+}S_{\rho}^{0}N)}.
\]
Using this together with \eqref{eq:ineq-induction}, we obtain
\[
\begin{split}
    \|m_{k}\|_{L^{2}(N)} &\leq C_{0} \|M_{k}\|_{H^{1/2}(\partial_{+}S_{\rho}^{0}N)}+C_{0}\sum_{j=0}^{k-1}\binom{k}{j}C_{k-j}A_{j} \sum_{r=0}^{j}\|M_{r}\|_{H^{1/2}(\partial_{+}S_{\rho}^{0}N)} \\
    & \leq C_{0}\left( 1+\sum_{j=0}^{k-1}\binom{k}{j}C_{k-j}A_{j} \right) \sum_{r=0}^{k}\|M_{r}\|_{H^{1/2}(\partial_{+}S_{\rho}^{0}N)}, 
\end{split}
\]
completing the proof.    
\end{proof}

\appendix

\section{Microlocal analysis for Gevrey regularity} \label{appendix:gevrey}

Let $\sigma \geq 1$. For $\sigma>1$, the validity of the implicit/inverse function theorems in the $G^\sigma$ setting is given in \cite{ko} and that locally finite $G^\sigma$ partitions of unity subordinate to open covers exist follows from \cite{HormanderI}*{Theorem~1.4.10}.

We now state
\begin{thm}\label{thm:doublefibFIO}
Let $\sigma\geq 1$ and $U , V$ be oriented $G^\sigma$ manifolds without boundary of dimensions $n_U, n_V$ respectively and $Z \subset U\times V$ an oriented embedded $G^\sigma$-submanifold of dimension $n_U+n'$ with $n_V = n'+n'', n', n'' >0$ so that $\pi_U \colon Z \to U$ is a submersion. Fix nowhere-vanishing $G^\sigma$ orientation forms $\omega_Z$ on $Z$ and $\omega_U$ on $U$, and let $d\omega_{U_y}$ be the $G^\sigma$ orientation form on $U_y = \pi_V(\pi_U^{-1}(y))$ induced by the submersion $\pi_U$ (which exists due to \cite{MST23}*{Lemma~2.1}) with $\pi_V\colon Z \to V$ the projection. Assume also that for every compact $K\subset V$ the map $\pi_U \colon Z\cap \pi_V^{-1}(K)\to U$ is proper. Let $\kappa \in G^\sigma(U\times V)$, and define the operator
\[
    Tu(y) = \int_{U_y} \kappa(y,x) u(x) d \omega_{U_y}(x)\,,\quad y\in U\,, u\in C_c^\infty(V)\,.
\]
The operator $T$ is a Fourier integral operator of order $n_V/4 - n'/2-n_U/4$ with canonical relation $(N^* Z\setminus 0)'$ whose kernel can locally be written as 
\[
    T(y,x) = \int_{\R^{n''}
} e^{i(\phi(y,x')-x'')\cdot \eta} t(y,x) d \eta\,,
\]
where $\phi,t\in G^\sigma$. Here $Z$ is locally written as $\{x'' = \phi(y,x')\}$ as in \cite{MST23}*{Lemma~2.3}.
\end{thm}
Observe that the properness assumption on $\pi_U$ is holds if either $Z$ is closed or $\pi_V$ is a proper submersion. 

For the analytic case $\sigma=1$, we refer to the statement at the beginning of \cite{MST23}*{\S~2}. For $\sigma>1$, the proof \cite{MST23}*{Theorem~2.2} applies mutatis mutandis replacing ``smooth'' with $G^\sigma$, where we use the fact that volume forms on $G^\sigma$-manifolds are locally $G^\sigma$ functions multiplied by the Lebesgue measure, and \cite{MST23}*{Lemma~2.3} is valid in the $G^\sigma$ setting by application of the $G^\sigma$ inverse function theorem, and we may reduce to local arguments by using $G^\sigma$ partitions of unity. 

\section{Stability of the magnetic and \texorpdfstring{$\MP$}{MP}-ray transforms} \label{appendix:magnetic}

\subsection{Magnetic systems} \label{subsection:magnetic}

In this section we recall the theory of magnetic systems and prove a stability estimate for the magnetic ray transform. For further details, we refer to \cite{DPSU07}.

A \emph{magnetic system} consists of a smooth compact Riemannian manifold with smooth boundary $(N,h)$ together with a closed 2-form $\Omega$, which is called the \emph{magnetic field}. Given a magnetic system $(N,h,\Omega)$, the magnetic field $\Omega$ induces the \emph{Lorentz force} $Y \colon TN \to TN$ given by
\[ 
\Omega_{p}(u,v)=(Y_{p}u,v)_{h}, \quad p \in N, \, u,v \in T_{p}N.
\]
$C^{2}$ curves $\sigma \colon [a,b] \to N$ that satisfy
\begin{equation} \label{eq:mag-geo}
    \nabla_{\dot{\sigma}} \dot{\sigma}=Y(\dot{\sigma})
\end{equation}
are called \emph{magnetic-geodesics}. Here, $\nabla$ is the Levi-Civita connection associated to the metric $h$. Equation \eqref{eq:mag-geo} defines a flow, called the \emph{magnetic flow}, and is given by
\[ 
\varphi_{t}(p,v)=(\sigma(t),\dot{\sigma}(t)),  
\]
where $\sigma$ solves \eqref{eq:mag-geo} and $\sigma(0)=p$, $\dot{\sigma}(0)=v$. This is a flow in the unit sphere bundle $SN$ of $N$ (i.e., $|\dot{\sigma}|_{h}=1$). Let $\tilde{\nu}(p)$ be the inward unit vector normal to $\partial N$ at $p$, and denote by $\Pi$ the second fundamental form of $h$. We say that $\partial N$ is \emph{strictly magnetic-convex} if
\[
\Pi(p, v)>(Y_p(v), \tilde{\nu}(p))_{h}, \quad \forall (p, v) \in S(\partial N).
\]
For $p \in N$, the \emph{magnetic-exponential map} at $p$ is the map 
\[
\exp _p^{\mu} \colon T_p N \to N\,, \quad \exp _p^{\mu}(t v)=\pi \circ \varphi_t(v),
\]
where $t \geq 0$, $v \in S_p N$, and $\pi \colon T N \to N$ is the base point projection.

\begin{defin}
We say that $N$ is \emph{magnetically simple} with respect to $(h, \Omega)$ if $\partial N$ is strictly magnetic-convex and that $(\exp_{p}^{\mu})^{-1}$ is a diffeomorphism for every $p \in N$.   
\end{defin}

In this case, $\Omega$ is exact, that is, there exists a 1-form $\alpha$ such that $\Omega=d\alpha$, and we call $\alpha$ the magnetic potential. Henceforth we call $(h, \alpha)$ a simple magnetic system on $N$. We will also say that $(N, h, \alpha)$ is a simple magnetic system.

We recall that $\partial_{+}SN=\{(p,v) \in SN: p\in \partial N, (v,\tilde{\nu}(p))_{h}\geq 0\}$ is the \emph{influx boundary}, where $\nu_{h}$ is the inward unit normal. For $(p, v) \in \partial_{+} SN$, let $\tau(p, v)$ be the exit time function for the magnetic-geodesic $\sigma$ with $\sigma(0)=p$, $\dot{\sigma}(0)=v$. By \cite{DPSU07}, we have that for a simple magnetic system, the function $\tau \colon \partial_{+} SN \to \R$ is smooth.

Given $f \in C(SN)$, we define its \emph{magnetic X-ray transform} over the magnetic geodesic $\sigma \colon [0,\tau(p,v)] \to N$ by
\[
I_{\mu}f(\sigma)=\int_{\sigma}f=\int_{0}^{\tau(p,v)} f(\sigma(t),\dot{\sigma}(t))dt,
\]
where $(\sigma(0),\dot{\sigma}(0))=(p,v)$. For simple magnetic systems, each magnetic geodesic can be identified with a unique $(p,v) \in \partial_{+}SN$. Hence, we will write $I_{\mu}f(p,v)$ instead of $I_{\mu}f(\sigma)$. 

We will prove

\begin{thm} \label{thm:stability-magnetic-xray}
Let $(N, h,\alpha)$ be a simple magnetic system. If $I_{\mu}$ is injective, there exists a constant $C>0$ such that for all $f \in L^2(N)$,
\[
\|f\|_{L^2(N)}\leq C\|I_{\mu} f\|_{H^{1 / 2}(\partial_{+}SN_{1})}.
\]
\end{thm}

This can be thought of as the magnetic analog of the stability estimate for the geodesic X-ray transform appearing in \cite{AS20}, \cite{SUbook}*{\S~VI.4}. We will need the following preliminary result, which, in the case of the geodesic ray transform, also appears in \cite{AS20} as a consequence of \cite{SU12}*{Proposition 5.12}.

\begin{lemma} \label{lemma:mapping-prop-magnetic}
Assume that $(N_{1},h_{1},\alpha_{1})$ is an extension of $(N,h,\alpha)$ so that it is non-trapping and strictly convex. Then, for every $s \geq 0$
\[  
I_{\mu} \colon H_{N}^{s-1/2} \to H^{s}(\partial_{+}SN_{1}), \quad I_{\mu}^{*}\colon H^{s-1/2}(\partial_{+}SN_{1}) \to H_{\mathrm{loc}}^{s}(N_{1}),
\]
are continuous. Here $(N_{1}, g_{1},\alpha_{1})$ is a simple system of the same dimension so that $M \Subset M_{1}$, and $g=g_{1}|_M$, $\alpha=\alpha_{1}|_{M}$. We keep the notation $g$ for $g_1$, and $\alpha$ for $\alpha_{1}$. We extend the function from $N$ to $N_{1}$ by zero.
\end{lemma}

\begin{proof}
The proof is an adaptation of \cite{SUbook}*{Theorem~VI.4.8.}. The argument is local on the set of magnetic geodesics. Fix a magnetic geodesic
$\gamma_{0}$ starting at $(x_{0},v_{0})$, and let $\mathcal{H} \subset \partial_{+}SN_{1}$ be a small coordinate neighborhood of $(x_{0},v_{0})$. We localize the transform to $\mathcal{H}$ by using a partition of unity, and write $y=(y^{1},\dots,y^{2n-2})$ for local coordinates on $\mathcal{H}$. It is enough to prove the estimate for $s=k\in \N$, since the general case follows by interpolation. Let $f\in C_{0}^\infty(N)$. For $|\alpha|\leq k$, the derivative \(\partial_y^\alpha I_\mu f\) is a finite sum of magnetic ray transforms whose weights are obtained by differentiating the magnetic flow with respect to the initial data. Equivalently, these derivatives are expressed in terms of magnetic Jacobi fields along the corresponding magnetic geodesics. Thus each term has the form $I_{\mu,w_{\beta}}(P_{\beta} f)$ where $w_{\beta}$ is a smooth compactly supported weight and $P_{\beta}$ is a differential operator of order $|\beta|\leq |\alpha|$, obtained from differentiation in the base variable along the magnetic Jacobi fields. Then, $I_{\mu}^*\partial_{y}^{\alpha} I_{\mu}$ is a finite sum of operators of the form $P_{\beta}^{*} I_{\mu,w_{\beta}}^{*} I_{\mu,\tilde{w}_{\beta}}$. Since the magnetic system is non-trapping and has strictly magnetic-convex boundary, we have $I_{\mu,w_{\beta}}^{*} I_{\mu,\tilde{w}_{\beta}} \in \Psi^{-1}$, see \cite{DPSU07}. Thus $I_{\mu}^{*}\partial_{y}^{\alpha} I_{\mu}\in \Psi^{|\alpha|-1}(N_{1})$. Now take $\chi \in C_{0}^{\infty}(N_{1})$ with $\chi \equiv 1$ in a neighborhood of $N$. We have
\[
\begin{split}
\|I_{\mu}f\|_{H^{k}}^{2}&=\sum_{|\alpha|\leq k}\|\partial_{y}^{\alpha}I_{\mu}f\|_{L^{2}}^{2} \\
&\leq \sum_{|\alpha|\leq 2k}|\langle I_{\mu}^{*}\partial_{y}^{\alpha}I_{\mu}f,f \rangle_{L^{2}}| \\
&=\sum_{|\alpha|\leq 2k}|\langle (\mathrm{Id}-\Delta)^{-(|\alpha|-1)/2}\chi I_{\mu}^{*}\partial_{y}^{\alpha}I_{\mu}f,(\mathrm{Id}-\Delta)^{(|\alpha|-1)/2}f\rangle_{L^{2}}|  \\
&\leq C\sum_{|\alpha|\leq 2k}\|f\|_{H^{(|\alpha|-1)/2}}^{2}.
\end{split}
\]
Summing over a partition of unity on \(\partial_+SN_1\), we obtain the result for $s=k$. The second property follows from a similar argument and duality.
\end{proof}

\begin{proof}[Proof of Theorem \ref{thm:stability-magnetic-xray}]
Recall from \cite{DPSU07} that $I_{\mu}^{*}I_{\mu}$ is an elliptic $\Psi$DO of order $-1$, which is elliptic since the magnetic system is simple. Furthermore, $I_{\mu}$ is injective in $L^{2}$. Thus, there exists a constant $C>0$ such that for any $f \in L^{2}(N)$
\[ \|f\|_{L^{2}(N)} \leq C\|I_{\mu}^{*}I_{\mu}f\|_{H^{1}(N_{1})}.\]
The result now follows from the mapping properties in Lemma \ref{lemma:mapping-prop-magnetic}.     
\end{proof}

\subsection{\texorpdfstring{$\MP$}{MP}-systems} \label{subsection:mp}

In this last subsection, we recall basic facts about $\MP$-systems and we prove the stability of the $\MP$-ray transform used in \S\ref{section:stability}. We begin by giving several definitions, and refer to \cites{AZ15, MT23, MT24} for further details.

An \emph{$\MP$-system} (magnetic-potential) consists of a smooth compact magnetic system $(N,h,\Omega)$, endowed with a smooth function $U$ called the \emph{potential}. $C^{2}$ curves $x \colon [a,b] \to N$ that satisfy
\begin{equation} \label{eq:mp-geo}
    \nabla_{\dot{x}} \dot{x}=Y(\dot{x})-\nabla U(x),
\end{equation}
are called \emph{$\MP$-geodesics}. Equation \eqref{eq:mp-geo} defines a flow, called the \emph{$\MP$-flow}, and is given by
\[ \phi_{t}(x,v)=(x(t),\dot{x}(t)),  \]
where $x$ solves \eqref{eq:mp-geo} and $x(0)=p$, $\dot{x}(0)=v$. For the $\MP$-flow, the energy $E \colon TN \to \R$ given by $E(x,v)=\frac{1}{2}|v|_{h}^{2}+U(x)$ is an integral of motion. Curves satisfying \eqref{eq:mp-geo}  with $E(x,\dot{x}) \equiv k$ are called $\MP$-geodesics of energy $k$. Given $k \in \R$, we define $S^{k}N:=\{E=k\}$. We will always assume that $k>\max_{N} U$. Let $\tilde{\nu}(p)$ be the inward unit vector normal to $\partial N$ at $p$, and set
\[
\partial_{\pm} S^{k}N:=\{(p, v) \in S^{k}N: p \in \partial N, \pm (v, \tilde{\nu}(p))_{h(p)} \geq 0\}.
\]
We denote by $\Pi$ the second fundamental form of $h$. We say that $\partial N$ is \emph{strictly $\MP$-convex} if
\[
\Pi(p, v)>(Y_p(v), \tilde{\nu}(p))_{h}-d_x U(\nu(p)) \quad \forall (p, v) \in S^k(\partial N).
\]
For $p \in N$, the $\MP$-exponential map at $p$ of energy $k$ is the map given by
\[
\exp _p^{k} \colon T_p N \to N \quad \exp _p^{k}(t v)=\pi \circ \phi_t(v),
\]
where $t \geq 0, v \in S_p^k N$, and $\pi \colon T N \to N$ is the base point projection.

\begin{defin} 
We say that $N$ is $k$-($\MP$-)\emph{simple} with respect to $(h, \Omega, U)$ if $\partial N$ is strictly $\MP$-convex and the $\MP$-exponential map $\exp_{p}^{k} \colon (\exp_p^{k})^{-1}(N) \to N$ is a diffeomorphism for every $p \in N$.   
\end{defin}

Note that this definition depends on the energy level.

As for magnetic case, when the $\MP$-system is simple, $\Omega$ is exact, that is, there exists a 1-form $\alpha$ such that $\Omega=d\alpha$, and we call $\alpha$ to be the magnetic potential. Henceforth we call $(h, \alpha, U)$ a simple $\MP$-system on $N$. We will also say that $(N, h, \alpha, U)$ is a simple $\MP$-system.

For $(p, v) \in \partial_{+} S^{k}N$, let $\tau(p, v)$ be the exit time function for the $\MP$-geodesic $\sigma$ with $\sigma(0)=p$, $\dot{\sigma}(0)=v$. By \cite{AZ15}*{Lemma~A.3} we have that for a simple $\MP$-system, the function $\tau \colon \partial_{+} S^{k}N \to \R$ is smooth.

Let $(N,h,\alpha,U)$ be a simple $\MP$-system. For $f \in C(S^{k}N)$, we define the \emph{$\MP$-ray transform} of $f$ by
\[ 
I_{\MP}f(x)=\int_{x}f:=\int_{0}^{\tau(p,v)}f(x(t),\dot{x}(t))dt,
\]
where $x \colon [0,\tau(p,v)] \to N$ is any geodesic of energy $k$ with $(x(0),\dot{x}(0))=(p,v)$. As in the case of magnetic systems, for simple $\MP$-systems, each $\MP$-geodesic can be identified with a unique $(p,v) \in \partial_{+}S^{k}N$. Hence, we will write $I_{\MP}f(p,v)$ instead of $I_{\MP}f(x)$.

\begin{defin}
Given a simple $\MP$-system $(N,g,\alpha,U)$ of energy $k$, we associate to it the magnetic system $(N,2(k-U)g,\alpha)$ (of energy $1/2$), which we call \emph{reduced magnetic system}.    
\end{defin}

By \cite{AZ15} we know that a $\MP$-system is simple if and only if its magnetic reduction is simple.

Finally, we will prove a stability estimate for $I_{\MP}$:

\begin{thm} \label{thm:MP-stability}
Let $(N,h,d(\rho\omega),-\rho^2/(2\lambda))$ be a simple $\MP$-system at energy $k$. There exists a constant $C_{0}$ depending on $k$ and $\rho$ so that
    \[
    \|u\|_{L^2(N)} \leq C_{0} \|I_{\MP} u\|_{H^{1 / 2}(\partial_{+} S^{k} N)}.
    \]
\end{thm}

\begin{proof}
Since the $\MP$-system is simple, its magnetic reduction is simple too. Hence, $I_{\mu}$ is injective by the results in \cite{DPSU07}. Thus, Theorem \ref{thm:stability-magnetic-xray} is valid. Let us denote the associated magnetic ray transform by $I_{\mu}$. Consider the smooth diffeomorphism
\[
    \Phi:\partial_{+}S^{k}N\to \partial_{+}SN,
    \qquad
    \Phi(X,V_x)
    =
    \left(X,\frac{V_{x}}{2k+\rho^{2}/\lambda}\right).
\]
By \cite{MT24}*{Proposition~3.7}, $I_{\mathcal{MP}}u=(I_\mu (u/(2k+\rho^{2}/\lambda)))\circ\Phi$. Furthermore, since $\Phi$ is a smooth diffeomorphism between compact smooth manifolds with boundary, pullback by $\Phi$ is an isomorphism. Hence there exists a constant $C_\Phi>0$ such that
\[
    \|F\|_{H^{1/2}(\partial_{+}SN)}
    \leq
    C_\Phi
    \|F\circ\Phi\|_{H^{1/2}(\partial_{+}S^{k}N)}
\]
for every $F\in H^{1/2}(\partial_+SN)$. Applying this with $F=I_\mu u$ and using that $N$ is compact gives
\[
    \|I_\mu u\|_{H^{1/2}(\partial_+SN)} \leq C_\Phi
    \|(I_\mu u)\circ\Phi\|_{H^{1/2}(\partial_{+}S^{k}N)}=C_\Phi
    \|I_{\mathcal{MP}}u\|_{H^{1/2}(\partial_{+}S^{k}N)},
\]
and the result follows from Theorem \ref{thm:stability-magnetic-xray}.
\end{proof}

\end{document}